\documentclass[leqno,11pt]{article}
\usepackage{amssymb,amsmath}
\usepackage{amsthm}
\usepackage{amsmath}
\newtheorem{thm}{Theorem}[section]
\newtheorem{cor}[thm]{Corollary}

\theoremstyle{definition}

\newtheorem{exa}[thm]{Example}

\numberwithin{equation}{section} \textheight  22 true cm \textwidth  15 true cm  
\def\a{\alpha}
\def\b{\beta}

\def\w{\widetilde}

\def\R{{\bf R}}

\begin{document}
\title{\bf The $\mathcal{L}$-sectional curvature of $S$-manifolds}
\author{Mehmet Akif AKYOL \\
Department of Mathematics, Faculty of Arts and Sciences \\
Bingol University, 12000 Bing\"{o}l, Turkey\\
makyol@bingol.edu.tr \and Luis M. FERN\'{A}NDEZ\footnote{The last two authors are partially supported by the PAI group FQM-327 (Junta de Andaluc\'{\i}a,
Spain, 2011) and by the
MEC project MTM 2011-22621 (MEC, Spain, 2011).} \\
Departmento de Geometr\'{\i}a y Topolog\'{\i}a\\
Facultad de Matem\'{a}ticas, Universidad de Sevilla\\
Apartado de Correos 1160, 41080 Sevilla, Spain\\
lmfer@us.es \and Alicia PRIETO-MART\'{I}N \\
Departmento de Geometr\'{\i}a y Topolog\'{\i}a\\
Facultad de Matem\'{a}ticas, Universidad de Sevilla\\
Apartado de Correos 1160, 41080 Sevilla, Spain\\
aliciaprieto@us.es}
\date{}
\maketitle

\begin{abstract}
We investigate $\mathcal{L}$-sectional curvature of $S$-manifolds with respect to the Riemannian connection and to certain semi-symmetric metric and
non-metric connections naturally related with the structure, obtaining conditions for them to be constant and giving examples of $S$-manifolds in
such conditions. Moreover, we calculate the scalar curvature in all the cases.
\end{abstract}


\noindent{\bf Keywords and sentences:} $S$-manifold, Semi-Symmetric Metric Connection, Semi-Symmetric non-Metric Connection, $\mathcal{L}$-Sectional
Curvature, Scalar Curvature.

\section{Introduction.}

In 1963, Yano \cite{Y} introduced the notion of $f$-structure on a $C^{\infty }$ $(2n+s)$-dimensional manifold $M$, as a non-vanishing tensor field
$f$ of type $(1,1)$ on $M$ which satisfies $f^3+f =0$ and has constant rank $r=2n$. Almost complex ($s=0$) and almost contact ($s=1$) are well-known
examples of $f$-structures. The case $s=2$ appeared in the study of hypersurfaces in almost contact manifolds \cite{BL,GY2} and it motivated that, in
1970, Goldberg and Yano \cite{GY} defined globally framed $f$-structures (also called $f$.pk-structures), for which the subbundle $\ker f$ is
parallelizable. Then, there exists a global frame $\{\xi_1,\dots,\xi_s\}$ for the subbundle $\ker f$ (the vector fields $\xi_1,\dots,\xi_s$ are
called the structure vector fields), with dual 1-forms $\eta^1,\dots,\eta^s$.

Thus, we can consider a Riemannian metric $g$ on $M$, associated with a globally framed $f$-structure, such that
$g(fX,fY)=g(X,Y)-\sum_{\a=1}^s\eta^\a(X)\eta^\a(Y)$, for any vector fields $X,Y$ in $M$ and then, the structure is called a metric $f$-structure.
Therefore, $TM$ splits into two complementary subbundles Im$f$ (whose differentiable distribution is usually denoted by $\mathcal{L}$) and $\ker f$
and, moreover, the restriction of $f$ to Im$f$ determines a complex structure.

A wider class of globally framed $f$-manifolds (that is, manifolds endowed with a globally framed $f$-structure) was introduced in \cite{B} by Blair
according to the following definition: a metric $f$-structure is said to be a $K$-structure if the fundamental 2-form $\Phi$, given by
$\Phi(X,Y)=g(X,fY)$, for any vector fields $X$ and $Y$ on $M$, is closed and the normality condition holds, that is,
$[f,f]+2\sum_{\a=1}^sd\eta^{\a}\otimes\xi_{\a}=0$, where $[f,f]$ denotes the Nijenhuis torsion of $f$. A $K$-manifold is called an $S$-manifold if
$d\eta^{\a}=\Phi$, for all $\a=1,\dots,s$. If $s=1$, an $S$-manifold is a Sasakian manifold. Furthermore, $S$-manifolds have been studied by several
authors (see, for example, \cite{B1,CFF,HOA,KT}).

It is well known that there are not exist $S$-manifolds ($s\geq 2$) of constant sectional curvature and, for Sasakian manifolds, the unit sphere is
the only one. This is due to the fact that $K(X,\xi_\a)=1$ and $K(\xi_\a,\xi_\b)=0$, for any unit vector field $X\in\mathcal{L}$ and any
$\a,\b=1\dots,s$. For this reason, it is interesting to study the sectional curvature of planar sections spanned by vector fields of $\mathcal{L}$
(called $\mathcal{L}$-sectional curvature) and to obtain conditions for this sectional curvature to be constant.

Further, in 1924 Friedmann and Schouten \cite{FS} introduced semi-symmetric linear connections on a differentiable manifold. Later, Hayden \cite{H}
defined the notion of metric connection with torsion on a Riemannian manifold. More precisely, if $\nabla$ is a linear connection in a differentiable
manifold $M$, the torsion tensor $T$ of $\nabla $ is given by $T(X,Y)=\nabla _{X}Y-\nabla _{Y}X-[X,Y]$, for any vector fields $X$ and $Y$ on $M$. The
connection $\nabla $ is said to be symmetric if the torsion tensor $T$ vanishes, otherwise it is said to be non-symmetric. In this case, $\nabla $ is
said to be a semi-symmetric connection if $T(X,Y)=\eta(Y)X-\eta (X)Y$, for any $X,Y$, where $\eta$ is a 1-form on $M$. Moreover, if $g$ is a
Riemannian metric on $M$, $\nabla $ is called a metric connection if $\nabla g=0$, otherwise it is called non-metric. It is well known that the
Riemannian connection is the unique metric and symmetric linear connection on a Riemannian manifold. Recently, $S$-manifolds endowed with a
semi-symmetric either metric or non-metric connection naturally related with the $S$-structure have been studied in \cite{ATF2,ATF}.

In this paper, we investigate $\mathcal{L}$-sectional curvature of $S$-manifolds with respect to the Riemannian connection and to the semi-symmetric
metric and non-metric connections introduced in \cite{ATF2,ATF}, obtaining conditions for them to be constant and giving examples of $S$-manifolds in
such conditions. Moreover, we calculate the scalar curvature in all the cases.

\section{Preliminaries on $S$-manifolds.}\label{sec1}

A $(2n+s)-$ dimensional differentiable manifold $M$ is called a \textit{metric $f$-manifold} if there exist a $(1,1)$ type tensor field $f $, $s$
vector fields $\xi_1,\dots,\xi_s$, called {\it structure vector fields}, $s$ 1-forms $\eta^1,\dots,\eta^s$ and a Riemannian metric $g$ on $M$ such
that
\begin{equation}  \label{def1}
f^{2}=-I+\overset{s}{\underset{\a=1}{\sum }}\eta^{\a}\otimes \xi _{\a},\mbox{ }\eta ^{\a}(\xi _{\b})=\delta _{\a\b},\mbox{ }f\xi_{\a}=0,\mbox{
}\eta^{\a}\circ f=0,
\end{equation}
\begin{equation}  \label{def2}
g(fX,fY)=g(X,Y)-\overset{s}{\underset{\a=1}{\sum }}\eta ^{\a}(X)\eta ^{\a}(Y),
\end{equation}
for any $X,Y\in\mathcal{X}(M),$ $\a,\b\in\{1,\dots,s\}$. In addition:
\begin{equation}  \label{def3}
\eta ^{\a}(X)=g(X,\xi _{\a}),\mbox{ }g(X,fY)=-g(fX,Y).
\end{equation}

Then, a 2-form $\Phi $ is defined by $\Phi(X,Y)=g(X,fY)$, for any $X,Y\in\mathcal{X}(M)$, called the \textit{fundamental 2-form}. In what follows, we
denote by $\mathcal{M}$ the distribution spanned by the structure vector fields $\xi_1,\dots,\xi_s$ and by $\mathcal{L}$ its orthogonal complementary
distribution. Then, $\mathcal{X}(M)=\mathcal{L}\oplus\mathcal{M}$. If $X\in\mathcal{M}$, then $fX=0$ and if $X\in\mathcal{L}$, then $\eta^\a(X)=0$,
for any $\a\in\{1,\dots,s\}$, that is, $f^2X=-X$.

In a metric $f$-manifold, special local orthonormal basis of vector fields can be considered: let $U$ be a coordinate neighborhood and $E_1$ a unit
vector field on $U$ orthogonal to the structure vector fields. Then, from (\ref{def1})-(\ref{def3}), $fE_1$ is also a unit vector field on $U$
orthogonal to $E_1$ and the structure vector fields. Next, if it is possible, let $E_2$ be a unit vector field on $U$ orthogonal to $E_1$, $fE_1$ and
the structure vector fields and so on. The local orthonormal basis $\{E_1,\dots,E_n,fE_1,\dots,fE_n,\xi_{1},\dots,\xi_{s}\},$, so obtained is called an
\textit{$f$-basis}.

Moreover, a metric $f$-manifold is \textit{normal} if
\begin{equation*}
[f,f]+2\underset{\a=1}{\overset{s}{\sum}}d\eta^{\a}\otimes\xi _{\a}=0,
\end{equation*}
where $[f,f]$ denotes the Nijenhuis tensor field associated to $f$. A metric $f$-manifold is said to be an \textit{$S$-manifold} if it is normal and
\begin{equation*}
\eta ^{1}\wedge\cdots\wedge\eta^s\wedge(d\eta ^{\a})^{n}\neq 0\mbox{ and }\Phi=d\eta ^{\a},\mbox{ }1\leq \a\leq s.
\end{equation*}

Observe that, if $s=1$, an $S$-manifold is a Sasakian manifold. For $s\geq 2$, examples of $S$-manifolds can be found in \cite{B,B1,HOA}.

If $\nabla$ is a linear connection on an $S$-manifold and $K$ denotes the sectional curvature associated with $\nabla$, the {\it
$\mathcal{L}$-sectional curvature} $K_\mathcal{L}$ of $\nabla$ is defined as $K_\mathcal{L}(X,Y)=K(X,Y)$, for any $X,Y\in\mathcal{L}$. The {\it
scalar curvature} of the $S$-manifold with respect to $\nabla$ is given by
\begin{equation}\label{scalar}
\tau=\frac{1}{2}\sum_{i,j=1}^{2n+s}K(e_i,e_j),
\end{equation}
for any local orthonormal frame $\{e_1,\dots,e_{2n+s}\}$ of tangent vector fields to $M$.

\section{The $\mathcal{L}$-sectional curvature of $S$-manifolds.}

From now on, let $M$ denote an $S$-manifold $(M,f,\xi_1,\dots,\xi_s,\eta^1,\dots,\eta^s,g)$ of dimension $2n+s$. We are going to study the sectional
curvature of $M$ with respect to different types of connections on $M$.

\subsection{The case of the Riemannian connection.}

First, let $\nabla$ denote the Riemannian connection of $g$. For the sectional curvature $K$ of $\nabla$, in \cite{CFF} it is proved that
\begin{equation}\label{sectnabla}
K(\xi_{\a},X)=R(\xi_\a,X,X,\xi_\a)=g(fX,fX),
\end{equation}
for any $X\in\mathcal{X}(M)$ and $\a\in\{1,\dots,s\}$. Consequently, if $s=1$, the unit sphere is the only Sasakian manifold of constant (sectional)
curvature. If $s\geq 2$, from (\ref{sectnabla}), we deduce that $M$ cannot have constant sectional curvature. For this reason, it is necessary to
introduce a more restrictive curvature. In general, a plane section $\pi$ on a metric $f$-manifold $M$ is said to be an \textit{$f$-section} if it is
determined by a unit vector $X$, normal to the structure vector fields and $fX$. The sectional curvature of $\pi$ is called an \textit{$f$-sectional
curvature}. An $S$-manifold is said to be an \textit{$S$-space-form} if it has constant $f$-sectional curvature $c$ and then, it is denoted by
$M(c)$. The curvature tensor field $R$ of $M(c)$ satisfies \cite{KT}:
\begin{equation}\label{sccte}
R(X,Y,Z,W)=\sum_{\a,\b=1}^s\{g(fX,fW)\eta^\a(Y)\eta^\b(Z)
\end{equation}
\begin{equation*}
-g(f X,f Z)\eta^\a(Y)\eta^\b(W)+g(fY,fZ)\eta^\a(X)\eta^\b(W)
\end{equation*}
\begin{equation*}
-g(fY,fW)\eta^\a(X)\eta^\b(Z)\}
\end{equation*}
\begin{equation*}
+\frac{c+3s}{4}\{g(fX,fW)g(fY,fZ)-g(fX,fZ)g(fY,fW)\}
\end{equation*}
\begin{equation*}
+\frac {c-s}{4}\{\Phi(X,W)\Phi(Y,Z)-\Phi(X,Z)\Phi(Y,W)-2\Phi(X,Y)\Phi(Z,W)\},
\end{equation*}
for any $X,Y,Z,W\in\mathcal{X}(M)$.

Therefore, if $M$ is an $S$-space-form of constant $f$-sectional curvature $c$ and considering an $f$-basis, from (\ref{sectnabla}) and
(\ref{sccte}), we deduce that the scalar curvature of $M$ with respect to the curvature tensor field of the Riemanian connection $\nabla$ satisfies:
$$\tau=\frac{n(n-1)(c+3s)}{2}+n(c+2s).$$

Now, in view of (\ref{sectnabla}) it is interesting to investigate the conditions for $K_{\mathcal{L}}$ to be constant. In this context, we observe
that, if $n=1$, $K_\mathcal{L}$ is actually the $f$-sectional curvature. Moreover, for $n\geq 2$, we can prove the following theorem.

\begin{thm}\label{teoK}
Let $M$ be a $(2n+s)$-dimensional $S$-manifold with $n\geq 2$. If the $\mathcal{L}$-sectional curvature $K_{\mathcal{L}}$ with respect to the
Riemannian connection $\nabla$ is constant equal to $c$, then $c=s$. In this case, the scalar curvature of $M$ is:
$$\tau=ns(2n+1).$$
\end{thm}
\begin{proof}
It is clear that if $K_{\mathcal{L}}$ is constant equal to $c$, then $M$ is an $S$-space-form $M(c)$. Consequently, from (\ref{sccte}), we have
\begin{equation}\label{KLnabla}
K_{\mathcal{L}}(X,Y)=\frac{c+3s}{4}+\frac{3(c-s)}{4}g(X,fY)^2,
\end{equation}
for any orthonormal vector fields $X,Y\in\mathcal{L}$. Now, since $n\geq 2$, we can choose $X$ and $Y$ such that $g(X,fY)=0$. Thus, from
(\ref{KLnabla}) we deduce
$$\frac{c+3s}{4}=c,$$
that is, $c=s$.

Now, considering a local orthonormal frame of tangent vector fields such that $e_{2n+\alpha}=\xi_\alpha$, for any $\alpha=1,\dots,s$, since
$K(e_i,e_j)=K_\mathcal{L}(e_i,e_j)=s$, $i,j=1,\dots,2n$, $i\neq j$, and using (\ref{sectnabla}) and (\ref{scalar}), we get the desired result for the
scalar curvature.
\end{proof}

By using (\ref{sccte}) and (\ref{KLnabla}), we have:
\begin{cor}
Let $M(c)$ be an $S$-space-form of constant $f$-sectional curvature $c$. Then, $M$ is of constant $\mathcal{L}$-sectional curvature (equal to $c$) if
and only if $c=s$
\end{cor}

\begin{exa}\label{exa}
Let us consider $\R^{2n+2+(s-1)}$ with coordinates
$$(x_1\dots,x_{n+1},y_1,\dots,y_{n+1},z_1,\dots,z_{s-1})$$
and with its standard $S$-structure of constant $f$-sectional curvature $-3(s-1)$, given by (see \cite{HOA}):
$$\xi_{\a}=2\frac{\partial}{\partial z_{\a}},\mbox{ } \eta^{\a}=\frac{1}{2}\left( dz_{\a}-\sum_{i=1}^{n+1}y_{i}dx_{i}\right),\mbox{ }\a=1,\dots,s-1,$$
$$g=\sum_{\alpha =1}^{s-1}\eta ^{\a}\otimes \eta ^{\a}+\frac{1}{4}\sum_{i=1}^{n+1}(dx_{i}\otimes dx_{i}+dy_{i}\otimes dy_{i}),$$
$$fX=\sum_{i=1}^{n+1}(Y_{i}{\frac{\partial }{\partial x_{i}}}-X_{i}{\frac{\partial }{\partial y_{i}}})+\sum_{\alpha
=1}^{s-1}\sum_{i=1}^{n+1}Y_{i}y_{i}{\frac{\partial }{\partial z_{\alpha }}},$$ where
$$X=\sum_{i=1}^{n+1}(X_{i}{\frac{\partial }{\partial x_{i}}}+Y_{i}{\frac{\partial }{\partial y_{i}}})+\sum_{\a=1}^{s-1}Z_{\a}{\frac{\partial }
{\partial z_{\alpha }}}$$
is any vector field tangent to $\R^{2n+2+(s-1)}$.

Now, let $S^{2n+1}(2)$ be a $(2n+1)$-dimensional ordinary sphere of radius 2 and $M=S^{2n+1}(2)\times\R^{s-1}$ a hypersurface of $\R^{2n+2+(s-1)}$.
Let
$$\xi_s=\sum_{i=1}^{n+1}\left(-y_i\frac{\partial}{\partial x_i}+x_i\frac{\partial}{\partial
y_i}\right)-\sum_{i=1}^{n+1}\sum_{\alpha=1}^{s-1}y_i^2\frac{\partial}{\partial z_\alpha}$$ and $\eta^s(X)=g(X,\xi_s)$, for any vector field $X$
tangent to $M$. Then, if we put
$$\w\xi_\a=s\xi_\a;\mbox{ }\w\eta^\a=\frac{1}{s}\eta^\a;\mbox{ }\a=1,\dots,s;$$
$$\w f=f;\mbox{ }\w g=\frac{1}{s}g+\frac{1-s}{s^2}\sum_{\alpha=1}^s\eta^\a\otimes\eta^\a,$$
it is known (\cite{HOA}) that $(M,\w f,\w\xi_1,\dots,\w\xi_s,\w\eta^1,\dots,\w\eta^s,\w g)$ is an $S$-space-form of constant $f$-sectional curvature
$c=s$. Moreover, from (\ref{sccte}), it is easy to show that the $\mathcal{L}$-sectional curvature $K_\mathcal{L}$ is also constant and equal to $s$.
\end{exa}

\subsection{The case of a semi-symmetric metric connection.}

In \cite{ATF2}, a semi-symmetric metric connection on $M$, naturally related to the $S$-structure, is defined by
\begin{equation}
\nabla^*_{X}Y=\nabla_{X}Y+\underset{j=1}{\overset{s}{\sum }}\eta ^{j}(Y)X-\underset{j=1}{\overset{s}{\sum }}g\left( X,Y\right) \xi _{j}, \label{defn}
\end{equation}
for any $X,Y\in \mathcal{X}(M)$. For the sectional curvature $K^*$ of $\nabla^*$, the following theorem was proved in \cite{ATF2}:

\begin{thm}\label{secn}
Let $M$ be an $S$-manifold. Then, the sectional curvature of $\nabla^*$ satisfies
\begin{enumerate}
\item[(i)] $K^*(X,Y)=K(X,Y)-s$;
\item[(ii)] $K^*(X,\xi_\a)=K^*(\xi_\a,X)=2-s$;
\item[(iii)] $K^*(\xi_\a,\xi_\b)=K^*(\xi_\b,\xi_\a)=2-s$,
\end{enumerate}
for any orthonormal vector fields $X,Y\in\mathcal{L}$ and $\a,\b\in\{1,\dots,s\}$, $\a\neq\b$.
\end{thm}

Therefore, from Theorem \ref{teoK}, if $s\neq 2$, an $S$-manifold cannot have constant sectional curvature with respect to the semi-symmetric metric
connection defined in (\ref{defn}). For $s=2$, $M=S^{2n+1}(2)\times\R$ endowed with the connection $\nabla^*$ and the $S$-structure given in Example
\ref{exa} is an $S$-manifold of constant sectional curvature (equal to 0) with respect to $\nabla^*$. Moreover, for any $s$, by using Theorem
\ref{teoK} again and $(i)$ of Theorem \ref{secn}, if the $\mathcal{L}$-sectional curvature associated with $\nabla^*$ is constant equal to $c$, then
$c=0$ and examples of such a situation are given in Example \ref{exa}. In this case, the scalar curvature is given by:
$$\tau^*=\frac{(4ns+s(s-1))(2-s)}{2}.$$

Regarding the $f$-sectional curvature of $\nabla^*$, from Theorem 4.5 in \cite{ATF2}, we know that it is constant if and only if the $f$-sectional
curvature associated with the Riemannian connection is constant too. In this case, if $c$ denotes the constant $f$-sectional curvature of the
Riemannian connection, $c-s$ is the constant $f$-sectional curvature of $\nabla^*$. Furthermore, from $(i)$ of Theorem \ref{secn} and (\ref{KLnabla})
it is easy to show that
$$K^*_\mathcal{L}(X,Y)=\frac{c-s}{4}(1+3g(X,fY)^2),$$
for any orthonormal vector fields $X,Y\in\mathcal{L}$. Therefore, considering an $f$-basis, we deduce that the scalar curvature of a
$(2n+s)$-dimensional $S$-manifold of constant $f$-sectional curvature $c$ with respect to $\nabla^*$ satisfies:
$$\tau^*=\frac{n(n+1)(c-s)+(4ns+s(s-1))(2-s)}{2}.$$

\subsection{The case of a semi-symmetric non-metric connection.}

In \cite{ATF}, a semi-symmetric non-metric connection on $M$, naturally related to the $S$-structure, is defined by
\begin{equation*}
\w\nabla_{X}Y=\nabla_{X}Y+\underset{j=1}{\overset{s}{\sum }}\eta ^{j}(Y)X,
\end{equation*}
for any $X,Y\in \mathcal{X}(M)$. To consider the sectional curvature of $\w\nabla$ has no sense because $\w R(\xi_\a,X,X,\xi_\a)=1$, while $\w
R(X,\xi_\a,\xi_\a,X)=2$, for any unit vector field $X\in\mathcal{L}$ and any $\a\in\{1,\dots,s\}$ (see \cite{ATF} for the details). However, for the
$\mathcal{L}$-sectional curvature $\w K_\mathcal{L}$, we have that $\w K_\mathcal{L}(X,Y)=K_\mathcal{L}(X,Y)$, for any orthogonal vector fields
$X,Y\in\mathcal{L}$. Consequently, Theorem \ref{KLnabla} and Example \ref{exa} can be applied here. In the case of constant $\mathcal{L}$-sectional
curvature (equal to $s$) and since $\w R(\xi_\a,\xi_\b,\xi_\b,\xi_\a)=1$, for any $\a,\b\in\{1,\dots,s\}$, $\a\neq\b$, the scalar curvature is given
by:
$$\w\tau=2ns(n+1)+\frac{s(s-1)}{2}.$$

Regarding the $f$-sectional curvature of $\w\nabla$, in \cite{ATF} it is proved that it is constant if and only if the $f$-sectional curvature
associated with the Riemannian connection is constant too. In this case, both constant are the same and the curvature tensor field of $\nabla$ is
completely determined by $c$. Furthermore, since from (\ref{KLnabla}),
$$\w K_\mathcal{L}(X,Y)=\frac{c+3s}{4}+\frac{3(c-s)}{4}g(X,fY)^2,$$
for any orthonormal vector fields $X,Y\in\mathcal{L}$, considering an $f$-basis, we deduce that the scalar curvature of a $(2n+s)$-dimensional
$S$-manifold of constant $f$-sectional curvature $c$ with respect to $\w\nabla$ satisfies:
$$\w\tau=\frac{n(n+1)(c+3s)+s(s-1)}{2}.$$

\end{document}